\documentclass[12pt]{article}

\usepackage{mathrsfs}
\usepackage[utf8]{inputenc}
\usepackage[T2A]{fontenc}
\usepackage[english]{babel}
\usepackage{xcolor}
\usepackage{amsmath}
\usepackage{amsfonts, amssymb}
\usepackage{amsthm,epsfig,epstopdf,titling,array}
\usepackage{float}
\usepackage{soul}
\usepackage[initials]{amsrefs}
\usepackage{fullpage}
\usepackage{color}
\usepackage{eucal}

\usepackage[hidelinks]{hyperref}

\theoremstyle{plain}
\newtheorem{thm}{Theorem}[section]
\newtheorem{lemma}[thm]{Lemma}

\newtheorem*{claim*}{Claim}

\newtheorem{prop}[thm]{Proposition}
\newtheorem*{thm*}{Theorem}
\newtheorem*{lemma*}{Lemma}
\theoremstyle{definition}
\newtheorem{defn}{Defintion}[section]

\newtheorem{rem}[thm]{Remark}

\newcommand{\arrow}{\overrightarrow}

\newcommand{\wt}{\widetilde}

\newcommand{\supp}{\operatorname{supp}}
\author{S.T. Krymskii}
\title{On the lowest possible dimension of supports of solutions to the discrete Schrodinger equation}

\begin{document}
\maketitle
\begin{abstract}
In this article we study the possible size of support of solutions to the discrete stationary Schrodinger equation $\Delta u(x)+V(x)u(x)=0$ in $\mathbb{Z}^d$. We show that for any nonzero solution to any discrete stationary Schrodinger equation the dimension of the support is at least $\log_2(d)-7.$ 

In the related setting of $\mathbb{Z}_2$-valued harmonic functions in $\mathbb{Z}^d$ one can improve the estimate on support's dimension to $\log_2(d).$ 

However, we also provide an example where a $\mathbb{Z}_2$-valued harmonic function in $\mathbb{Z}^d$ has a fractal-like support with dimension $\log_2(d)+1$. This fractal satisfies a recurrence relation: $$X = 2X+\{e_1,-e_1,\dots,e_d,-e_d\}.$$ 

This example and estimate provide an answer to the Malinnikova's question about the smallest size of set $X\subset\mathbb{Z}^d$ such that no cross contains exactly one point of $X$.
\end{abstract}
\textbf{Acknowledgements.} The estimate for the dimension of the support of discrete solutions to the Schrodinger equation was produced within the scope of the NCCR SwissMAP which was funded by the Swiss National Science Foundation (grant number 205607).  The example of $\mathbb{Z}_2$-valued harmonic function supported on a fractal subset of $\mathbb{Z}^d$ is a part of a master thesis of the author (supervised by A. Logunov who was funded by SNSF).
\section{Introduction}

The article of Bourgain and Kenig \cite{BK05} concerns a version of the Anderson-Bernoulli model in $\mathbb{R}^d$ and proves Anderson localisation near the edge of the spectrum in this model. One of the used tools  is quantitative unique continuation for elliptic PDEs in $\mathbb{R}^d$. One of the crucial difficulties to prove Anderson localization near the edge of the spectrum for $\mathbb{Z}^d$ models is the lack of theorems on discrete quantitative unique continuation. We refer to the talk of Charlie Smart \cite{BLMS22} where the exact problems (useful for Anderson localisation) on discrete quantitative unique continuation are described.  The localization near the edge of the spectrum in the discrete setting of Anderson-Bernoulli model has been established in the cases of $\mathbb{Z}^2$ and $\mathbb{Z}^3$  by Ding, Smart, Li and Zhang, see  \cite{DS20}, \cite{LZ22}, \cite{S22}, but is open in the higher dimensional cases.

We now describe one of the problems concerning discrete quantitative unique continuation. Consider a solution to discrete Schrodinger equation $\Delta u+Vu = 0$ in $\mathbb{Z}^d$: 
\begin{equation}
\label{eq:Sch}
\sum_{i=1}^d u(x\pm e_i) -2du(x)+V(x)u(x) = 0.
\end{equation}
What is the smallest possible size of $u$'s support?

To specify the size of the support, we define  the discrete dimension of a set $X$ as $$\limsup_{r\to\infty} \frac{\ln |X \cap B_r(0)|}{\ln r}.$$

Define a cross centered at $x$ as the set $\bigcup_i \{x+ e_i,x-e_i\}\cup\{x\}$. If a point $x\pm e_i$ is in the support of $u,$ then the discrete equation \eqref{eq:Sch} implies that the cross must contain at least 2 points from $\supp u$. Then one can give the following definition.
\begin{defn}
\label{satis}
    Call a set $X\subset\mathbb{Z}^d$ \textit{satisfying the cross condition} if there exists no point $y\in\mathbb{Z}^d$ such that the set $X$ intersects the cross $\{y\}\bigcup_{i=1}^d\{y-e_i,y+e_i\}$ in exactly one point. 
\end{defn}

\textbf{Question of Malinnikova.} Suppose that a set $X\subset\mathbb{Z}^d$ contains $0$ and satisfies the cross condition. Can the discrete dimension of $X$ be strictly smaller than $d/2$?

It is known (see, e.g. \cite{S22}) that for any even dimension $d$ there exists a function $u$ whose support has discrete dimension $d/2$ solving the Schrodinger equation with a constant potential  $$\sum_{i=1}^d u(x\pm e_i) = 0.$$ For $d=2$ the example of such a function is $u(x,y) = (-1)^x\delta_{xy}.$

The main results are the following. In Section \ref{2} we  present an example of a $\mathbb{Z}_2$-valued harmonic function in $\mathbb{Z}^d$, whose support has dimension $\log_2(2d).$ A modification of this example shows (see Remark \ref{MalinRem}) that there is a set satisfying the cross condition and having dimension $\log_2(2d).$ For $d\geq 9$ the dimension of the set $\log_2(2d)<d/2,$ and the modified example provides the answer to the Malinnikova question for $d\geq 9.$  Therefore, one cannot use the requirements mentioning only the supports to show that solutions to the discrete Schrodinger equation have support with discrete dimension at least $d/2.$ 

However, one can also show that harmonic functions with values in $\mathbb{Z}_2$ cannot have a support with discrete dimension less than $\log_2(d)$, as seen in Section \ref{3}. In addition, we prove in Section \ref{5} that any nonempty set $X$ such that no cross covers exactly one point of $X$ has discrete dimension at least $\log_2(d)-7.$ These results provide an asymptotically sharp estimate on discrete dimension of sets satisfying the cross condition from Malinnokova's question. 

\section{Example}
\label{2}
\begin{thm}
\label{2.1}
    There exists a $\mathbb{Z}_2$-valued  harmonic function $u$ defined on the set $\mathbb{Z}^d$ such that there is a constant $C = C(d)$ s.t. for any $n\in\mathbb{N}$ there is no ball of size $n$ containing more than $Cn^{\log_2(2d)}$ points of $\supp u.$
\end{thm}
\begin{rem}
\label{2.1alt}
    Another form of the statement is the following: there exist a set $X\subset \mathbb{Z}^d$ and a constant  $C = C(d)$ s.t. for any $n\in\mathbb{N}$ there is no ball of size $n$ containing more than $Cn^{\log_2(2d)}$ points of $X,$ and any point $y\in\mathbb{Z}^d$ has an even amount of neighbors in $X$. The latter condition on the set $X$ is true if and only if the characteristic function of the set $X$ is harmonic.
\end{rem}
Before proving the Theorem \ref{2.1}, we present several lemmas and use them to prove the Theorem \ref{2.1}. In addition, we define  \begin{equation}
    \label{S}
    S = x_1+x_1^{-1}+x_2+x_2^{-1}+\dots+x_d+x_d^{-1}.
\end{equation}
\begin{lemma}
Consider the ring $\mathbb{Z}[x_1,\dots,x_d,x_1^{-1},\dots,x_d^{-1}]$ of series over the ring $\mathbb{Z}.$ For any set $X\subset \mathbb{Z}^d$ define the formal series $P_X = \sum_{t\in X}\prod x_i^{t_i}.$ Then $$P_XS = \sum_{t\in\mathbb{Z}^d} N_X(t)\prod x_i^{t_i},$$ where $N_X(t)$ is the number of neighbors of $t$ in $X.$
\end{lemma}
\begin{proof}
    Trivial.
\end{proof}
\begin{rem}
\label{rem2.4}
For the ring $\mathbb{Z}_2[x_1,\dots,x_d,x_1^{-1},\dots,x_d^{-1}]$ of series over the ring $\mathbb{Z}_2$ we find that $$P_XS = \sum_{t\in\mathbb{Z}^d} (N_X(t)\bmod 2)\prod x_i^{t_i}.$$
\end{rem}

Consider the algebra of polynomials $\mathbb{Z}_2[x_1,x_1^{-1},\dots,x_d,x_d^{-1}].$ Since this algebra is over a field of characteristic 2, $S^{2^k} = \sum x_i^{\pm 2^k}.$ But what is $S^{2^k-1}?$ It turns out that $$S^{2^k-1} = SS^2S^4\dots S^{2^{k-1}} = \prod_{i=0}^{k-1} \sum_{j=1}^d x_j^{\pm 2^i}.$$ Therefore, $S^{2^k-1}$ is the sum $\sum_{(y_1,\dots,y_d)\in X_k} \prod x_i^{y_i},$ where \begin{equation}
    \label{Xk} X_k = \sum_{i=0}^k 2^i\{e_1,-e_1,e_2,-e_2,\dots,e_d,-e_d\}.
\end{equation} 
\begin{lemma}
\label{2.4}
For a point $x\in\mathbb{Z}^d$ the number of neighbors of $x$ in $X_k$ is odd if and only if $x = \pm 2^ke_i$ for a unit vector $e_i$.
\end{lemma}
\begin{proof}
Consider the polynomial $\sum_{(y_1,\dots,y_n)\in X_k} \prod x_i^{y_i} = \prod_{i=0}^{k-1} \sum_{j=1}^d x_j^{\pm 2^i}.$ Remember that $(\sum a_i)^2 = \sum a_i^2$ in any algebra with characteristics 2. Therefore, $$\sum_{j=1}^d x_j^{\pm 2^i} = \left(\sum_{j=1}^d x_j^{\pm 2^{i-1}}\right)^2 = \dots = \left(\sum_{j=1}^d x_j^{\pm 1}\right)^{2^i},$$ and $\prod_{i=0}^{k-1} \sum_{j=1}^d x_j^{\pm 2^i} = \left(\sum_{j=1}^d x_j^{\pm 1}\right)^{2^k-1} = S^{2^k-1}.$ 

However, due to Remark \ref{rem2.4} for any set $X$ $$S\sum_{(y_1,\dots,y_d)\in X} \prod x_i^{y_i}= \sum_{(z_1,\dots,z_d) \in X^*} \prod x_i^{z_i},$$ where $X^*$ is the set of points having an odd number of neighbors in $X$ and $S$ is given by the formula \eqref{S}. 

Finally, $$S\sum_{(y_1,\dots,y_d)\in X_k} \prod x_i^{y_i} = SS^{2^k-1}=S^{2^k} = \sum_{i=1}^d x_i^{\pm 2^k},$$ and $X_k^* = \{\pm 2^ke_i\}.$
\end{proof}

Next we show that the set $X_{n+1}$ defined by the formula \eqref{Xk} contains the set $X_n.$ Recalling that $2^n(\pm e_i) = 2^{n+1}(\pm e_i)+2^n(\mp e_i),$ we find that $$X_{n+1} = X_{n-1}+\{\pm 2^n e_i\}+\{\pm 2^{n+1}e_i\} =$$$$= X_{n-1}+\left(\{\pm 2^n e_i\}\cup\{\pm(3\cdot 2^n)e_i\}\cup\{\pm 2^{n+1}e_i\pm 2^ne_j|i\neq j\}\right)\supset$$$$\supset X_{n-1}+\{\pm 2^n e_i\}= X_n.$$ Therefore, $X_{n+1}$ contains $X_n.$  The following lemma shows that the elements which are in $X_{n+1}\setminus X_n$ are very far from zero. 
\begin{lemma}
\label{2.5}The sets $X_n$ defined by the equation \ref{Xk} have the property $$(X_n\setminus X_{n-1})\cap [-2^{n-1},2^{n-1}]^d = \emptyset.$$
\end{lemma}
\begin{proof}
Suppose that an element $x$ is in the set $X_n$, but not in $X_{n-1}.$ Then $x = \sum_{i=0}^n 2^iv_i,$ where the vectors $v_i$ are in the set $\{e_1,-e_1,e_2,-e_2,\dots,e_d,-e_d\},$ but $v_n\neq -v_{n-1}.$ Then $\sum_{i=0}^{n-2}2^iv_i$ has no coordinates whose absolute value is at least $2^{n-1}.$ However, the vector $2^{n-1}v_{n-1}+2^nv_n$ has a coordinate whose absolute value is at least $2^n.$ Thus $x$ has a coordinate whose absolute value is larger than $2^{n-1}.$
\end{proof}
\begin{proof}[Proof of Theorem \ref{2.1}]
Since $X_{k+1}\supset X_k$, but the new elements appear increasingly far away, as seen in Lemma \ref{2.5}, for any finite set $Y$ the intersection $X_k\cap Y$ is constant, starting from a certain $k.$ Define the set \begin{equation}
    \label{Xinf} X_\infty := \cup X_i = \bigcup_i \sum_{j=0}^i 2^j\{e_1,-e_1,\dots,e_d,-e_d\}.
\end{equation} Then $X_\infty\cap Y = \lim_{k\to\infty} (X_k\cap Y).$ For example, taking $Y = \{x\pm e_i\}$ for a certain point $x,$ we find out that the number of neighbors of $x$ in $X_\infty$ is the limit of numbers of neighbors in $X_k.$ 

But for any point $x$ the number of its neighbors in $X_k$ is odd only if $x$ is at distance $2^k$ from the zero, as seen in Lemma \ref{2.4}. Thus for any fixed point $x$ the number of neighbors in $X_k$ is even in the long run, and $X_\infty$ contains evenly many neighbors of $x:$
\begin{equation}
    \label{GoodXinf}
    2|\left|\phantom{\frac{1}{2}}X_\infty\cap\{x+e_1,x-e_1,\dots,x-e_d\}\right|.
\end{equation}

Next we'll estimate the number of elements in the cube $[-n;n]^d.$  Remember that $X_k\setminus X_{k-1}\cap [-2^{k-1},2^{k-1}]^d = \emptyset,$ and $$|X_\infty\cap [-2^{k},2^{k}]^d|\leq |X_k| \leq (2d)^{k+1}.$$ On the other hand, $|X_\infty\cap [-2^{k},2^{k}]^d|\geq |X_{k-1}|.$ The sizes of $X_{k}$ are estimated as follows. If $x\in X_k,$ then $x$ is in one of $2d$ copies of $X_{k-1}$: $X_{k-1}+2^{k-1}e_1,X_{k-1}-2^{k-1}e_1,...$ or $X_{k-1}-2^{k-1}e_d$. Thus $|X_k|=2d|X_{k-1}|=(2d)^k.$ Therefore, the cube $[-2^k;2^k]^d$ contains between $(2d)^k$ and $(2d)^{k+1}$ points, and the cube $[-n;n]^d$ contains between $(2d)^{[\log_2(n)]}$ and $(2d)^{2+[\log_2(n)]}$ points. But $(2d)^{\log_2(n)} = n^{\log_2(2d)},$ and the number of points from $X_{\infty}$ inside the cube $[-n;n]^d$ is comparable with $n^{\log_2(2d)}.$ Thus $X_\infty$ has discrete dimension $\log_2(2d) = 1+\log_2d$. Since any point $x$ has an even amount of neighbors in $X_\infty$ by construction (see equation \eqref{GoodXinf}), the characteristic function of the set $X_\infty$ is discrete $\mathbb{Z}_2$-valued harmonic (see Remark \ref{2.1alt}).
\end{proof}
\begin{rem}
\label{2.7}
    $X_\infty$ can also be represented recursively. Indeed, one has $$X_{i}=\sum_{j=0}^i2^j\{e_1,-e_1,\dots,e_d,-e_d\} = \{e_1,-e_1,\dots,e_d,-e_d\}+\sum_{j=1}^i2^j\{e_1,-e_1,\dots,e_d,-e_d\} =$$$$= \{e_1,-e_1,\dots,e_d,-e_d\}+2X_{i-1}.$$ Therefore,  \begin{equation}
        \label{recur} X_\infty = \{e_1,\dots,-e_d\}+2X_\infty = \bigcup_{i=1}^d (2X_\infty\pm e_i).
    \end{equation}
    This formula means that $X_\infty$ is the set of neighbors of $2X_\infty,$ and all points in $X_\infty$ have an odd sum of coordinates. It also means that for any unit vector $v_i$ the sets $X_\infty$ and $X_\infty+v_i$ have no common points.

    Moreover, the union in the formula \eqref{recur} is disjoint. Indeed, suppose that a point $x$ is both in $2X_\infty+ v_i$ and in $2X_\infty + v_j$ for two different unit vectors $v_i$ and $v_j.$ Then $x$ is equivalent mod 2 both to $v_i$ and to $v_j.$ Then $v_i=\pm v_j.$ It means that $v_i = -v_j.$ Therefore, $2X_\infty+2v_i$ and $2X_\infty$ have a common point $x+v_i$. Hence $X_\infty$ and $X_\infty+v_i$  have a common point $\frac{x+v_i}{2}$, which is impossible. This contradiction shows that the union \begin{equation}
        \label{disj}X_\infty = \bigsqcup_{i=1}^d(2X_\infty\pm e_i)
    \end{equation} is indeed disjoint.
\end{rem}
\begin{prop}
    \label{MalinRem}
    The example constructed in the proof of Theorem \ref{2.1} is an example of a discrete-valued harmonic function, but not of a function whose support satisfies the cross condition (see Definition \ref{satis}). However, the set $X_+ = X_\infty+\{0,e_1,-e_1,\dots,e_d,-e_d\},$ where $X_\infty$ is defined by the equation \eqref{Xinf}, satisfies the cross condition and has the same discrete dimension as $X_\infty$. 

\end{prop}
\begin{proof}
    Suppose that $X_+$ does not satisfy the cross condition. Then there is a point $x$ such that the intersection $X_+\cap \{x,x+e_1,x-e_1,\dots,x+e_d,x-e_d\}$ contains only one element. If $x$ is the only element in this intersection, then $x\in X_+.$ It either means that $x\in X_\infty$ (and $\{x,x+e_1,x-e_1,\dots,x+e_d,x-e_d\}\subset X_+$) or that one of the points $x\pm e_i$ is in $X_\infty$. But the latter implies that both points $x$ and $x\pm e_i$ are in $X_+.$

    Finally, suppose that $x\pm e_i$ is the only  point in $X_+\cap \{x,x+e_1,x-e_1,\dots,x+e_d,x-e_d\}$. Without loss of generality one can assume that this point is $x+e_1.$  Then one of the points $x+e_1\pm e_j$ is in $X_\infty.$ If $j\neq 1,$ then $x+e_1$ and $x\pm e_j$ are both adjacent to $x+e_1\pm e_j$ and therefore are both in $X_+.$ Then the cross centered at $x$ intersects $X_+$ at least in the two points $x+e_1$ and $x\pm e_j$, which contradicts our assumption that the intersection of the set $X_+$ and the cross centered at $x$ contains exactly 1 element. 
    
    Therefore, $j=1,$ and $x+2e_1\in X_\infty.$ But the set $\{x+e_1\pm e_j\}_{j=1}^d$ needs to contain an even amount of points from $X_\infty,$ while the set $\{x+e_1\pm e_j\}_{j=2}^d$ contains none of them. Therefore, $x$ and $x+2e_1$ are both in $X_\infty,$ which is also impossible.  Therefore, the set $X_+\cap \{x,x+e_1,x-e_1,\dots,x+e_d,x-e_d\}$ cannot contain exactly one point.
\end{proof}
\section{Optimality in the harmonic case}
\label{3}
Surprisingly,  the supports of harmonic functions with values in $\mathbb{Z}_2$ defined on $\mathbb{Z}^n$ cannot have discrete dimension lower than $\log_2(n)$. 

First of all, we consider the following metrics defined on $\mathbb{Z}^n:$ $$\delta((x_1,\dots,x_n),(y_1,\dots,y_n)) = \max|y_i-x_i|.$$ Define \begin{equation}
    \label{Nn}N_n(r):= \min_{u:\mathbb{Z}^n\to\mathbb{Z}_2, \Delta u = 0, u(0)=1} |\supp(u)\cap[-r;r]^n|.
\end{equation} This is the lowest possible number of points of $\supp u$ inside a cube of edge length $2r$ centered at a point in $\supp(u).$ 
\begin{thm}
\label{estim}The number $N_n(r)$, defined by the equation \eqref{Nn}, satisfies the inequality $N_n(r)\geq C(n)r^{\log_2(n)}.$ 
\end{thm} 
\begin{proof}
Consider a function $u$ whose support contains the origin. Denote $u$'s support by $X$ and consider the formal power series $$P(x_1,\dots,x_n) = \sum_{(d_1,\dots,d_n)\in X} \prod x_i^{d_i}.$$ Then $u$ is $\mathbb{Z}_2$-harmonic if and only if $SP = 0,$ where $S$ is defined by the equation \eqref{S}. 

Moreover, if $u$ is $\mathbb{Z}_2$-harmonic, then $S^2P = 0.$ Since $S = \sum x_i^{\pm 1},$ $S^2=\sum x_i^{\pm 2},$ and any function $$u_w(d_1,\dots,d_n) := u(2d_1+w_1,\dots,2d_n+w_n)$$ also turns out to be harmonic. This fact allows one to split the lattice $\mathbb{Z}^n$ into $2^n$ sublattices where the function $u$ is also harmonic: \begin{equation}
\label{longharm}
    \sum_{i=1}^n u_w(\arrow{d}+\arrow{e_i})+u_w(\arrow{d}-\arrow{e_i}) = 0.
\end{equation}

Let $r,R$ be arbitrary positive integers until we define them later. If $u$ is known to be nonzero in at least $C$ points $x_w$ that are located at distance at most $r$ from the zero and in pairwise different lattices, then for any positive integer $R$ $$N_n(r+R) = \sum_{w\in\{0;1\}^n} \left|\left\{x\in X|\delta(x,0)\leq r+R\right\}\bigcap \{x|x_i\equiv w_i\bmod 2\}\right| \geq $$$$\geq \sum_{w\in\{0;1\}^n} \left|\left\{x\in X|\delta(x,x_w)\leq R\right\}\bigcap \{x|x_i\equiv w_i\bmod 2\}\right| \geq CN_n(R/2).$$ In other words, \begin{equation}
    \label{NnC}
    N_n(r+R)\geq CN_n([R/2]).
\end{equation}
\begin{lemma}
\label{L3.1}
    For any harmonic function $u$ in the space $\mathbb{Z}^n$ with $n\geq 2$ such that $u(0,\dots,0)\neq 0$ there either exist at least $n$ points in $\supp(u)\cap B_r(0)$ with pairwise distinct remainders $\bmod\,2$ or we find that $|\supp(u)\cap B_r(0)|\geq rN_{n-1}([r/2]).$
\end{lemma}
\begin{proof}
    Consider the set $W$ of vectors $w\in\{0;1\}^n$ such that the support of $u$ contains at least one point in $(2\mathbb{Z})^n+w$ at distance at most $r$ from the zero.

    Call a point $w\in W$ \textit{bad} if there is a set $\{(w+e_i+e_1)\bmod 2,\dots,(w+e_i+e_n)\bmod 2\}$ containing only one point in $W.$ Then for any vector $v\in\mathbb{Z}^n$ such that the cross \\$\bigcup_{j=1}^n\{w+2v+e_i\pm e_j\}$ is inside $[-r;r]^n$ the cross can intersect $\supp u$ only by the points $w+2v$ and $w+2v+2e_i.$  But $u$ is harmonic, and the support of $u$ intersects the set  $\{w+2v+e_i\pm e_j\}_{j=1}^n$ by an even amount of points. Therefore, the support of $u$ contains $w+2v$ iff it contains $w+2v+2e_i.$

    Without loss of generality one can assume that $e_i=e_1.$ Then the function $u(w_1,w_2+2x_2,\dots,w_n+2x_n)$ is harmonic in $\mathbb{Z}^{n-1}\cap B_{r/2}(0)$ and  for any $x_1\in[-r/2;r/2]$ we have \\$u(w_1,w_2+2x_2,\dots) = u(w_1+2x_1,\dots).$ Then we find that $$\{x\equiv w\bmod 2|u(x)\neq 0, x\in[-r;r]^n\} = \left(\{w_1+2x_1\}\cap [-r;r]\right)\times \{(w_2+2x_2,\dots)|u(w_1,w_2+2x_2,\dots)\neq 0\}.$$ Therefore, the number of points in $B_r(0)\cap\supp u$ is at least $$r|\supp u \cap B_r(0)\cap \{(w_1,w_2+2x_2,\dots,w_n+2x_n)\}|\geq rN_{n-1}([r/2]).$$ 

Next we consider the case when there is no bad point $w\in W.$ If the point $(0,\dots,0)$ isn't bad, then the cross $\{e_1\}\cup\{e_1\pm e_i\}$ centered at $e_1$ contains the point $(0;\dots;0)\in\supp u$ and has to contain an extra point $e_1\pm e_i\in\supp(u).$  

    Without loss of generality one can assume that this point in $\supp u$ is $e_1+e_2.$ 

    Since $0$ isn't a bad point in $W,$ then for any $i>2$ there exists a number $j\neq i$ s.t. the set $W$ contains the point $e_i+e_j.$ Similarly, if $e_1+e_2$ isn't a bad point in $W$, then for any $i>2$ there exists a number $j\neq i$ s.t. the set $W$ contains the point $(e_1+e_2+e_i+e_j)\bmod 2.$  Call a point in $W$ \textit{connected} with 0 by the number $i$ if it has the form $e_i+e_j$ for some $j.$ Similarly we call a point in $W$ \textit{connected} with $e_1+e_2$ by the number $i$ if it has the form $(e_1+e_2+e_i+e_j)\bmod 2.$ Any point is connected with $0$ and $e_1+e_2$ by a total of 2 or less numbers. However, $0$ and $e_1+e_2$ are to be connected by at least $n-2$ numbers each. Therefore, if the points $0$ and $e_1+e_2$ both aren't bad, the set $W$ contains at least $n$ points, and $C\geq n.$ 
    \end{proof}

Consider the way to make $|\supp u\cap B_{R+r}(0)| = N_n(r+R).$ If the set $u$ is such that the set $W$ from the lemma has a bad point, \begin{equation}
        \label{bad}N_n(r+R)\geq rN_{n-1}([r/2]).
    \end{equation}

Otherwise the Lemma \ref{L3.1} ensures that there exist at least $n$ points in $\supp(u)\cap B_r(0)$ with pairwise distinct remainders $\bmod\,2$, and in the formula \ref{NnC} we have $C\geq n$.

\begin{lemma}
    For any positive integers $r,R$ \begin{equation}
    \label{curse}
    N_n(r+R)\geq\min(nN_n(R/2), rN_{n-1}([r/2])).
\end{equation}
\end{lemma}
\begin{proof}
    The Lemma \ref{L3.1} and the inequality \ref{NnC} ensure that for any positive integer $r,R$ we either have \begin{equation}
    N_n(r+R)\geq nN_n([R/2])
\end{equation} or
\begin{equation}
    N_n(r+R)\geq rN_{n-1}([r/2]).
\end{equation} Since one of the two inequalities is true, we find the statement of the lemma.
\end{proof}
Next we reformulate the Theorem \ref{estim} and use the recurrence relation \eqref{curse} to prove it by induction on $n.$
\begin{thm*}[\ref{estim}]
    For any dimension $n\in\mathbb{N}$ there exists a constant $C_n>0$ s.t. for any $M>0$ the number $N_n(M)$, defined by the equation \eqref{Nn}, satisfies the inequality $N_n(M)\geq C_nM^{\log_2(n)}.$ 
\end{thm*}

\begin{proof}[Finishing the proof of Theorem \ref{estim} by induction]
Suppose that we know that for any $r>0$ $N_{n-1}(r)\geq C_{n-1}r^{\log_2(n-1)}.$ Then we'll try to find a constant $C_n>0$  s.t. for any $M>0$ $N_{n}(M)\geq C_{n}M^{\log_2(n)}.$

The number $N_n(M)$ for a large positive integer $M$ is estimated as follows. Set $r := [M^{3n/(3n+1)}],$ $ R = M-r$ and apply the inequality \eqref{curse} to find that \begin{equation}
\label{recurs}
    N_n(M)\geq\min\left(nN_n(M/2 - M^{3n/(3n+1)}/2), [M^{3n/(3n+1)}]N_{n-1}([M^{3n/(3n+1)}/2])\right).
\end{equation}

First of all, we prove the induction base for $n=1$ and $n=2.$ Indeed, if $n=2,$ then for any point $x\in\supp u$ one of the points $x+e_1+ e_2,$ $x+e_1-e_2$ or $x+2e_1$ is in $\supp u.$ Consider the sequence $X_k$ defined as follows. $X_0=x,$ $X_{k+1}$ is a point in the set $$(X_k+\{e_1+e_2,e_1-e_2,2e_1\})\cap \supp u.$$ This set is obviously nonempty. Moreover, the point $X_{k+1}$ is located at distance 2 or less from $X_k$, and the point $X_k$ is located at distance at most $2k$ from $X_0.$ Therefore, if $k<r/2,$ the point $X_k$ is in the ball $B_r(x),$ this ball contains at least $r/2$ points, and the Theorem \ref{estim} is true for $n=2.$ For $n=1$ the Theorem yields $N_1(r)\geq C,$ which is trivial.

Next we prove the theorem for $n>2$ assuming that it's true for $n-1.$ Indeed, for any $r>0$ we have $N_{n-1}(r)\geq C_{n-1}r^{\log_2(n-1)},$ and the inequality \eqref{recurs} implies that  $$N_{n}(M)\geq \min\left(nN_{n}([M/2 - [M^{\frac{3n}{3n+1}}]/2]), [M^{3n/(3n+1)}]C_{n-1}[M^{\frac{3n}{3n+1}}]^{\log_2(n-1)}/n\right).$$ But for $n\geq 3$ $$\frac{1+\log_2(n-1)}{\log_2(n)}= 1+\frac{1-\log_2\frac{n}{n-1}}{\log_2(n)} \geq 1+\frac{1/3}{\log_2(n)}>1+\frac{1}{3n}= \frac{3n+1}{3n}.$$ Therefore, $M^{\frac{3n}{3n+1}}C_{n-1}M^{\frac{3n}{3n+1}\log_2(n-1)}/n\geq C_{n-1}M^{\log_2(n)}/n$, and $$[M^{3n/(3n+1)}]C_{n-1}[M^{\frac{3n}{3n+1}}]^{\log_2(n-1)}\geq C'M^{\log_2(n)}$$ for some $C'>0.$

The ratio $\frac{N_n(M)}{M^{\log_2(n)}}$ satisfies another inequality: $$\frac{N_n(M)}{M^{\log_2(n)}}\geq \min\left(\frac{nN_n([M/2 - [M^{\frac{3n}{3n+1}}]/2])}{M^{\log_2(n)}}, C'\right).$$ Obviously, for any $r\in(0;M)$ $M = 2\frac{M}{2} = 2\frac{M-r}{2}\frac{M}{M-r},$ and $$\frac{n}{M^{\log_2(n)}} = \frac{1}{(M/2)^{\log_2(n)}}= \frac{1}{((M-r)/2)^{\log_2(n)}}(1-\frac{r}{M})^{\log_2(n)}.$$ Finally, the two previous formulas imply that $$\frac{N_n(M)}{M^{\log_2(n)}}\geq \min\left(\frac{N_n([M/2 - [M^{\frac{3n}{3n+1}}]/2])}{([M/2 - [M^{\frac{3n}{3n+1}}]/2])^{\log_2(n)}}\left(1-M^{\frac{3n}{3n+1}-1}\right)^{\log_2(n)}, C'\right).$$

Consider the function $f(M) = \min(\frac{N_n(M)}{M^{\log_2(n)}},C').$ Then $$f(M) \geq \min(C',\min\left(\frac{N_n([M/2 - [M^{\frac{3n}{3n+1}}]/2])}{[M/2 - [M^{\frac{3n}{3n+1}}]/2]^{\log_2(n)}}\left(1-M^{\frac{3n}{3n+1}-1}\right)^{\log_2(n)}, C'\right))\geq$$$$\geq \min(C', f([M/2 - [M^{\frac{3n}{3n+1}}]/2])\left(1-M^{\frac{3n}{3n+1}-1}\right)^{\log_2(n)}) = f([M/2 - [M^{\frac{3n}{3n+1}}]/2])\left(1-M^{\frac{3n}{3n+1}-1}\right)^{\log_2(n)}.$$ 

Then $$\ln f(M)\geq \ln f([M/2-[M^{\frac{3n}{3n+1}}/2]]) - \frac{C\log_2(n)}{M^{\frac{1}{3n+1}}}.$$ Denote $X:= \frac{C\log_2(n)}{2^{1/(3n+1)}-1}.$ Then $C\log_2(n)+X = 2^{1/(3n+1)}X,$ and $$\ln f(M) - \frac{X}{M^{\frac{1}{3n+1}}}\geq \ln f([M/2-[M^{\frac{3n}{3n+1}}/2]]) - \frac{C\log_2(n)}{M^{\frac{1}{3n+1}}}-\frac{X}{M^{\frac{1}{3n+1}}}  =$$$$= \ln f([M/2-[M^{\frac{3n}{3n+1}}/2]]) - \frac{X}{(M/2)^{1/(3n+1)}}\geq \ln f([M/2-[M^{\frac{3n}{3n+1}}/2]])-\frac{X}{[M/2-[M^{\frac{3n}{3n+1}}/2]]^{1/(3n+1)}}.$$ Therefore, $\ln f(M) - \frac{X}{M^{\frac{1}{3n+1}}}$ cannot increase as we transition from $M$ to $[M/2-[M^{\frac{3n}{3n+1}}/2]].$ Thus $\ln f(M) - \frac{X}{M^{\frac{1}{3n+1}}}$ is bounded from below, and so is $\ln f(M).$ And $f(M)$ is bounded from below by a constant $C_n>0,$ meaning that $\frac{N_n(M)}{M^{\log_2(n)}}\geq C_n$ and that $$N_n(M)\geq C_nM^{\log_2(n)}$$ for any positive integer $M$. 
\end{proof}
\color{white}
\end{proof}
\color{black}

\section{Optimality for the Malinnikova question}
The example with the logarithmic dimension described above turns out to be almost optimal for the Malinnikova question as well. Indeed, one can prove that  any set $X$ which doesn't intersect any cross $\{x\}\cup\bigcup\{x\pm e_i\}$ by exactly one point cannot have discrete dimension much smaller than the one displayed by the example constructed in Section \ref{2}.
\label{5}
\begin{thm}
\label{mainmal}
    Call a set $X\subset\mathbb{Z}^d$ \textit{supportive} if no cross $\{x\}\cup\{x\pm e_i\}$ intersects $X$ only by one of the points $x\pm e_i$. Then any \textit{supportive} set has discrete dimension at least $\log_2(d)-7.$
\end{thm}
\begin{rem}
    Any set satisfying the cross condition (see Definition \ref{satis}) is supportive. In addition, for any function $u$ defined on $\mathbb{Z}^d$ with values in any field such that $\Delta u+Vu = 0$ we find that $\supp u$ is supportive. However, the condition of being supportive is weaker than the cross condition; for example, the set $X_\infty$ defined by the equation \eqref{Xinf} or \eqref{disj} is the support of a harmonic function with values in $\mathbb{Z}_2,$ but $X_\infty$ doesn't satisfy the cross condition, since any cross centered at a point from $X_\infty$ doesn't contain any other points from this set.  
\end{rem}
\subsection{Start of proof}
\begin{proof}[Proof of Theorem \ref{mainmal}]
In our proof we consider two
different metrics: the Manhattan distance $\|x-y\| = \sum|y_i-x_i|$ and the maximal-difference distance $\delta(x,y):= \max|y_i-x_i|.$
    \begin{lemma}
    \label{4.2}
        For any supportive set $X$, point $x\in X,$ positive integer $r$ and unit vector $e\in\mathbb{Z}^d$ there exists a point $y\in X$ such that $(y-x,e) \in\{r,r+1\}$ while the Manhattan distance between $y$ and $x+re$ is at most $r.$
    \end{lemma}
    \begin{proof}
        For any point $x\in X$ the cross centered at $x+e$ contains not just $x$, but also another point. Denote this other point by $P(x,e,1).$ This point is equal to $x+e,x+2e$ or $x+e\pm e'.$ Consider the points $$\wt{P}(x,e,n) = \begin{cases}x,& n=0\\P(\wt{P}(x,e,n-1),e,1),& n>0\end{cases}.$$ While $(\wt{P}(x,e,n)-x,e)<r$, we cannot increase the Manhattan distance to $x+re$ by moving from $\wt{P}(x,e,n)$ to $\wt{P}(x,e,n+1)$. Therefore, if we repeatedly move from $\wt{P}(x,e,n)$ to $\wt{P}(x,e,n+1),$ we'll encounter the necessary point when $(\wt{P}(x,e,n)-x,e)$ exceeds $r.$
    \end{proof}
    Next we define $P(x,e,r)$ to be a point $y$ mentioned in the lemma and define $A(x,e,r):= P(x,e,r)-x-re.$ We will use these points to construct lots of different points that are close to $x.$ Notice that \begin{equation}
        \label{scalars}
        (A(x,e,r), e)\in\{0;1\}
    \end{equation} and 
    \begin{equation}
        \label{dists}
        \|A(x,e,r)\|\leq r.
    \end{equation}
    Consider the $2d$ unit vectors $e_1,-e_1,\dots,e_d,-e_d.$ For $m\in\{1,\dots,2d\}$ and $i\in \{1,\dots,d\}$ we define $\epsilon_m:= (-1)^{m+1}e_{[(m+1)/2]}$ and $P_i(x,r):= P(x,\epsilon_i,r).$ Next we consider the sequence of radii $r_n = 2^{n+2}-2.$ For a point $x$ we consider a sequence $a = \{a_i\}_{i=0}^{n-1}$ of $n$ numbers within the set $\{1,\dots,2d\}.$ Denote the length of the sequence $a$ by $n(a)$ and consider the points $$\begin{cases}Q_0(x,a) := x\\Q_{m+1}(x,a) := P_{a_{m}}(Q_m(x,a),r_{n(a)-m})\end{cases}.$$ What we want is the existence of many different sequences $a$ of the same length $n(a)$ such that the points $Q_{n(a)}(x,a)$ rarely end up the same.

    \begin{lemma}
    \label{54}
        Call two integers $b$ and $c$ \textit{opposing} if they are different, but $[\frac{b+1}{2}] = [\frac{c+1}{2}].$ Then for any sequences $a^1$ and $a^2$ of length $n$ such that their first different elements are opposing, $Q_n(x,a^1)\neq Q_n(x,a^2).$  
    \end{lemma}
    \begin{proof}
        Without loss of generality one can assume that the first different element of the two sequences is $a^1_0\neq a^2_0.$ 
        
        The maximal difference distance $\delta$ between $Q_1(x,a^1)$ and $Q_1(x,a^2)$ is $2r_n$, while the maximal difference distance between $Q_l(x,a)$ and $Q_{1+l}(x,a)$ is either $r_{n-l}$ or $r_{n-l}+1.$ Since $r_n-r_{n-1}-\dots-r_0>n,$ we find that $Q_n(x,a^1)\neq Q_n(x,a^2).$ 
    
        %Similarly, if the sequences $a^1$ and $a^2$ are the same until the $i$th number, but the vectors $\epsilon_{a_i^1}$ and $\epsilon_{a_i^2}$ are opposite, the points $Q_n(x,a^1)$ and $Q_n(x,a^2)$ end up different. 
    \end{proof}
\begin{defn}
\label{kgood}
    A sequence $a$ is called $k$-good if for any $i,j$ such that $|i-j|<k$ the numbers $\lceil a_i/2\rceil$ and $\lceil a_j/2\rceil$ are different. 
\end{defn}
\begin{rem}
    For a $k$-good sequence $a$ and any $i,j$ such that $|i-j|<k$ the vectors $\epsilon_{a_i}$ and $\epsilon_{a_j}$ are orthogonal.
\end{rem}
    First of all, we fix a number $k$ that we'll specify later and suppose that there exist $l$ different $k$-good sequences $a^1,\dots,a^l$ of the same length $n = n(a^i)$ such that $$Q_n(x,a^1)=Q_n(x,a^2)=\dots=Q_n(x,a^l),$$ while the numbers $a^i_0$ are pairwise different.
    Lemma \ref{54} ensures that if the points $Q_n(x,a^i)$ coincide and the numbers $a^i_0$ are pairwise different, then the vectors $\epsilon_{a_0^i}$ are orthogonal.

    On the other hand, $$\|Q_n(x,a)-x\| \leq \sum_{m=1}^{n} \|Q_m(x,a)-Q_{m-1}(x,a)\|\leq 2r_n+2r_{n-1}+\dots+2r_1\leq 4r_n.$$ Since $Q_n(x,a^1)=\dots=Q_n(x,a^l),$ while the vectors $\epsilon_{a^1_0},\dots,\epsilon_{a^l_0}$ are orthogonal unit vectors, 
    \begin{equation}
        \label{prods}\sum_{i=1}^l |(Q_n(x,a^i)-x,\epsilon_{a^i_0})| \leq \|Q_n(x,a^i)-x\|\leq 4r_n.
    \end{equation} 
Define \begin{equation}
    \label{Psi}
    \Psi(x,a^i_0):= (Q_n(x,a^i)-x,\epsilon_{a^i_0}).
\end{equation} We have just shown that these expressions cannot simultaneously have large values. However, if $\Psi(x,a^i_0)$ is small, then any sequence $a'$ of length $n$ such that  $Q_n(x,a')=Q_n(x,a^i)$ must \textit{coincide} with $a^i$ for a long time, as we'll see below.
\subsection{The main lemma in the proof of Theorem \ref{mainmal}}
\begin{lemma}
\label{4.5}
    If $k\geq 3,$ the sequences $a$ and $a'$ have length $n$ and are $k$-good, $Q_n(x,a) = Q_n(x,a')$ and $a'_0=a_0,$ then the lowest number $m$ such that $a_m\neq a'_m$ satisfies the inequality $$m\geq -\log_2\left(\frac{2^{-n}+2^{2-k}+2\max(0,\Psi(x,a_0'))/r_n}{1-6\cdot 2^{-k}}\right)-1.$$ In other words, $a_m'=a_m$ for any $$m< -\log_2\left(\frac{2^{-n}+2^{2-k}+2\max(0,\Psi(x,a_0'))/r_n}{1-6\cdot 2^{-k}}\right)-1.$$
\end{lemma}
\begin{proof}
    First of all, \begin{equation}
    \label{Qx}
        Q_n(x,a)-x = \sum_{w=0}^{n-1} \epsilon_{a_w}r_{n-w}+\sum_{w=0}^{n-1} A(Q_w(x,a),r_{n-w},\epsilon_{a_w}).
    \end{equation} The former sum is such that its scalar product with $\epsilon_{a_0}$ is at least $(1-2^{1-k})r_n:$ \begin{equation}
    \label{epsilons}
    \left(\sum_{w=0}^{n-1} \epsilon_{a_w}r_{n-w}, \epsilon_{a_0}\right) = \sum (\epsilon_{a_w},\epsilon_{a_0})r_{n-w} \geq r_n+0r_{n-1}+\dots+0r_{n-k+1}-r_{n-k}-\dots \geq (1-2^{1-k})r_n.
    \end{equation} Therefore, the equations \eqref{Qx},\eqref{epsilons} and \eqref{Psi} imply that \begin{equation}
    \label{Ae}
        \sum_{w=0}^{n-1} \left(A(Q_w(x,a),r_{n-w},\epsilon_{a_w}),\epsilon_{a_0}\right)= (Q_n(x,a)-x,\epsilon_{a_0})-\left(\sum_{w=0}^{n-1} \epsilon_{a_w}r_{n-w}, \epsilon_{a_0}\right)\leq \Psi(x,a_0)-(1-2^{1-k})r_n.
    \end{equation} But $A(x,r_n,\epsilon_{a_0})$ is almost orthogonal to $\epsilon_{a_0}$ due to the equation \eqref{scalars}. Therefore, the equations \eqref{scalars} and \eqref{Ae} imply that \begin{equation}
    \label{21}
        \left(-\epsilon_{a_0},\sum_{w=1}^{n-1} A(Q_w(x,a),r_{n-w},\epsilon_{a_w})\right)\geq (1-2^{1-k})r_n-\Psi(x,a_0)-1.
    \end{equation}
    Define $Proj(v)$ to be the projection of the vector $v$ on the hyperplane orthogonal to $\epsilon_{a_0}$ and recall that $\|v\|$ is the Manhattan norm of the vector $v.$  For any vector $v$ $$\|Proj(v)\| = \|v\| - |(-\epsilon_{a_0},v)|.$$ If we sum those equalities for the vectors $ A(Q_w(x,a),r_{n-w},\epsilon_{a_w}),$ we find that \begin{multline}
    \label{23}
        \sum_{w=1}^{n-1}\left\|Proj\left( A(Q_w(x,a),r_{n-w},\epsilon_{a_w})\right)\right\| =\\= \sum_{w=1}^{n-1}\left\| A(Q_w(x,a),r_{n-w},\epsilon_{a_w})\right\|        
        -\sum_{w=1}^{n-1}\left|\left(-\epsilon_{a_0}, A(Q_w(x,a),r_{n-w},\epsilon_{a_w})\right)\right|.
    \end{multline} The equation \eqref{dists} ensures that $$\sum_{w=1}^{n-1}\left\| A(Q_w(x,a),r_{n-w},\epsilon_{a_w})\right\|\leq \sum_{w=1}^{n-1} r_{n-w}\leq  r_n,$$ while the equation \eqref{21} ensures that $$\sum_{w=1}^{n-1}\left|\left(-\epsilon_{a_0}, A(Q_w(x,a),r_{n-w},\epsilon_{a_w})\right)\right|\geq (1-2^{1-k})r_n-\Psi(x,a_0)-1.$$ Therefore, the equation \eqref{23} ensures that $$\sum_{w=1}^{n-1}\left\|Proj\left( A(Q_w(x,a),r_{n-w},\epsilon_{a_w})\right)\right\|\leq r_n-(1-2^{1-k})r_n+\Psi(x,a_0)+1 = 2^{1-k}r_n+\Psi(x,a_0)+1$$In other words,
    \begin{equation}
        \label{proji}
        \sum_{l=1}^{n-1} \|Proj(A(Q_l(x,a),r_{n-l},\epsilon_{a_l}))\|\leq \Psi(x,a_0)+2^{1-k}r_n+1.
    \end{equation}
    Similarly to the sequence $a,$ the same conclusion holds for the sequence $a'$: $$Q_n(x,a')-Q_1(x,a') = \sum_{l=1}^{n-1} \epsilon_{a_l'}r_{n-l}+\sum_{l=1}^{n-1} A(Q_l(x,a'),r_{n-l},\epsilon_{a_l'}).$$ Reasoning similar to the equations \eqref{21}---\eqref{proji} allows us to prove that
    \begin{equation}
        \label{projs}
        \sum_{w=1}^{n-1}\|Proj(A(Q_w(x,a'),r_{n-w},\epsilon_{a_w'}))\|\leq \Psi(x,a'_0)+2^{1-k}r_n+1.
    \end{equation}

    On the other hand, the sequences $a'$ and $a$ start with the same number $a'_0=a_0,$ and  $\epsilon_{a_0} = \epsilon_{a'_0}.$ In addition, the first point where the sequences $a$ and $a'$ differ is denoted by $m$. Then  $Q_m(x,a)=Q_m(x,a'),$ and $Q_{m+1}$ is the first point that differs for $a$ and $a'.$ The difference between $Q_n(x,a')$ and $Q_m(x,a')$ is defined by the formula $$Q_n(x,a')-Q_m(x,a') = \sum_{w=m}^{n-1} \epsilon_{a_w'}r_{n-w}+\sum_{w=m}^{n-1} A(Q_w(x,a'),r_{n-w},\epsilon_{a_w'}).$$ The  inequality \eqref{projs} lets one estimate the latter sum's projection on the hyperplane orthogonal to $\epsilon_{a'_0}$: $$\left\|Proj\left(\sum_{w=m}^{n-1} A(Q_w(x,a'),r_{n-w},\epsilon_{a_w'})\right)\right\|\leq \sum_{w=1}^{n-1}\left\|Proj\left( A(Q_w(x,a'),r_{n-w},\epsilon_{a_w'})\right)\right\|\leq \Psi(x,a_0')+2^{1-k}r_n+1.$$
    
    Define \begin{equation}
    \label{24}
        R:=Q_m(x,a)+\sum_{w=m}^{n-1}\epsilon_{a_w}r_{n-w},
    \end{equation}
    \begin{equation}
    \label{25}
        R':=Q_m(x,a')+\sum_{w=m}^{n-1}\epsilon_{a_w'}r_{n-w}.
    \end{equation} Then project $R,R',Q_n(x,a)$ onto the same hyperplane orthogonal to $\epsilon_{a_0}.$ Since $$Q_n(x,a') - R' =\sum_{w=m}^{n-1} A(Q_w(x,a'),r_{n-w},\epsilon_{a_w'}),$$ the  inequality \eqref{projs} implies that $$\|Proj(R'-Q_n(x,a'))\|\leq \Psi(x,a'_0)+2^{1-k}r_n+1.$$ Since the sequences $a$ and $a'$ start with the same number and since $Q_n(x,a)=Q_n(x,a'),$ we have $\Psi(x,a_0)=\Psi(x,a'_0)$, and the inequality \eqref{proji} ensures that $$\|Proj(R-Q_n(x,a))\|\leq \Psi(x,a'_0)+2^{1-k}r_n+1.$$ Therefore, the projections of $R'$ and $R$ are at distance $2\Psi(x,a'_0)+2^{2-k}r_n+2$ or less from each other.

    However, the points $Q_m(x,a)$ and $Q_m(x,a')$ are the same because the sequences $a$ and $a'$ have the same first $m$ members. Therefore, 
    \begin{equation}
        \label{between}\left\|Proj\left(\sum_{w=m}^{n-1}\epsilon_{a_w}r_{n-w}-\sum_{w=m}^{n-1}\epsilon_{a_w'}r_{n-w}\right)\right\| = \left\|Proj\left(R-R'\right)\right\|\leq 2\Psi(x,a'_0)+2^{2-k}r_n+2.
    \end{equation}
    Next we consider two cases. 
    
    \textbf{Case 1: $m<k$.} Then the vector $\epsilon_{a_0}$ is orthogonal to $\epsilon_{a_m}$ and to $\epsilon_{a_m'}.$ Consider the scalar products $\left(\sum_{w=m}^{n-1}\epsilon_{a_w}r_{n-w},\epsilon_{a_m}\right)$ and $\left(\sum_{w=m}^{n-1}\epsilon_{a_w'}r_{n-w},\epsilon_{a_m}\right).$ Similarly to the inequality \eqref{epsilons}, we find that 
    \begin{equation}
        \label{20'}\left(\sum_{w=m}^{n-1}\epsilon_{a_w}r_{n-w},\epsilon_{a_m}\right)\geq r_{n-m}(1-2^{1-k}).
    \end{equation} The second product $\left(\sum_{w=m}^{n-1}\epsilon_{a_w'}r_{n-w},\epsilon_{a_m}\right)$ is estimated differently. We know that $a_m\neq a_m'.$ It means that the vectors $\epsilon_{a_m}$ and $\epsilon_{a_m'}$ are either opposite or orthogonal. Since the points $Q_n(x,a)$ and $Q_n(x,a')$ coincide, the first case is impossible due to Lemma \ref{54}. Therefore, $\epsilon_{a_m}\perp\epsilon_{a'_m},$ and $$\left(\sum_{w=m}^{n-1}\epsilon_{a_w'}r_{n-w},\epsilon_{a_m}\right) =\left(\sum_{w=m+1}^{n-1}\epsilon_{a_w'}r_{n-w},\epsilon_{a_m}\right).$$ The sequence $a'$ is $k$-good. Therefore, if $(\epsilon_{a_w'},\epsilon_{a_m})\neq 0,$ then for $i\in[1;k-1]$ we have $(\epsilon_{a_{w+i}'},\epsilon_{a_m})= 0.$ Rewrite the scalar product $\left(\sum_{w=m+1}^{n-1}\epsilon_{a_w'}r_{n-w},\epsilon_{a_m}\right)$ as $$\left(\sum_{w=m+1}^{n-1}\epsilon_{a_w'}r_{n-w},\epsilon_{a_m}\right)  = \sum_{w=m+1}^{n-1} r_{n-w}(\epsilon_{a_w'},\epsilon_{a_m}).$$ Suppose that the first nonzero summand in the latter sum is at $w=m+j\geq m+1.$ Then the summands with $w=m+j+1,\dots,m+j+k-1$ are zero, and $$\sum_{w=m+1}^{n-1} r_{n-w}(\epsilon_{a_w'},\epsilon_{a_m})\leq r_{n-m-j}+r_{n-m-j-k}+\dots.$$ But $r_{n-m-j}\leq r_{n-m}/2,$ while $$r_{n-m-j-k}+\dots\leq 2^{-k-1}r_{n-m}+2^{-k-2}r_{n-m}+\dots = 2^{-k}r_{n-m}.$$ Therefore, $$\left(\sum_{w=m+1}^{n-1}\epsilon_{a_w'}r_{n-w},\epsilon_{a_m}\right) = \sum_{w=m+1}^{n-1} r_{n-w}(\epsilon_{a_w'},\epsilon_{a_m})\leq r_{n-m}(1/2+2^{-k}).$$ This inequality and the inequality \eqref{20'} imply that the difference between the products $\left(\sum_{w=m}^{n-1}\epsilon_{a_w}r_{n-w},\epsilon_{a_m}\right)$ and $\left(\sum_{w=m}^{n-1}\epsilon_{a_w'}r_{n-w},\epsilon_{a_m}\right)$ is at least $r_{n-m}(\frac{1}{2}-3\cdot 2^{-k}).$ Similarly, the difference between $\left(\sum_{w=m}^{n-1}\epsilon_{a_w}r_{n-w},\epsilon_{a_m'}\right)$ and $\left(\sum_{w=m}^{n-1}\epsilon_{a_w'}r_{n-w},\epsilon_{a_m'}\right)$ is at least $r_{n-m}(\frac{1}{2}-3\cdot 2^{-k}).$ 
    
    Since $m<k,$ the plane orthogonal to $\epsilon_{a_0}$ contains the vectors $\epsilon_{a_m}$ and $\epsilon_{a_m'}.$ 
    
    $$\|Proj\left(\sum_{w=m}^{n-1}\epsilon_{a_w}r_{n-w}-\sum_{w=m}^{n-1}\epsilon_{a_w'}r_{n-w}\right)\|\geq $$$$\geq\left(\sum_{w=m}^{n-1}\epsilon_{a_w}r_{n-w}-\sum_{w=m}^{n-1}\epsilon_{a_w'}r_{n-w}, \epsilon_{a_m}\right)+\left(\sum_{w=m}^{n-1}\epsilon_{a_w'}r_{n-w}-\sum_{w=m}^{n-1}\epsilon_{a_w}r_{n-w}, \epsilon_{a_m'}\right) \geq$$$$\geq r_{n-m}(1/2-3\cdot 2^{-k})\cdot 2 = r_{n-m}(1-6\cdot 2^{-k}).$$  But the norm of the same projection is also at most $2\Psi(x,a'_0)+2^{2-k}r_n+2$ due to the inequality \eqref{between}.      
    
    Therefore, $$r_{n-m}(1-6\cdot 2^{-k})\leq 2\Psi(x,a'_0)+2^{2-k}r_n+2.$$ Since $r_n = 2^{n+2}-2$ and $r_{n-m}=2^{n-m+2}-2,$ we have $$\frac{r_{n-m}}{r_n}= \frac{2^{n-m+2}-2}{2^{n-m+2}}2^{-m}\frac{2^{n+2}}{2^{n+2}-2} \geq 2^{-m-1},$$ and $$2^{-m-1}(1-6\cdot 2^{-k})\leq 2\frac{\Psi(x,a'_0)}{r_n}+2^{2-k}+2^{-n}.$$ Since $k\geq 3,$ we have $1-6\cdot 2^{-k}>0,$ and the first number $m$ such that $a_m\neq a'_m$ is $$m\geq -\log_2\left(\frac{2^{-n}+2^{2-k}+2\Psi(x,a_0')/r_n}{1-6\cdot 2^{-k}}\right)-1.$$ 

    \textbf{Case 2: $m\geq k$.} Then consider the expression $$-\log_2\left(\frac{2^{-n}+2^{2-k}+2\max(0,\Psi(x,a_0'))/r_n}{1-6\cdot 2^{-k}}\right).$$ But  $$\left(\frac{2^{-n}+2^{2-k}}{1-6\cdot 2^{-k}}\right)\geq \left(\frac{2^{2-k}}{1}\right) = 2^{2-k}.$$ Therefore, $$-\log_2\left(\frac{2^{-n}+2^{2-k}+2\max(0,\Psi(x,a_0'))/r_n}{1-6\cdot 2^{-k}}\right)\leq -\log_2\left(2^{2-k}\right)=k-2<k.$$ 
    
    If $m\geq k,$ we definitely have $$m\geq-\log_2\left(\frac{2^{-n}+2^{2-k}+2\max(0,\Psi(x,a_0'))/r_n}{1-6\cdot 2^{-k}}\right)-1.$$ In either case, the Lemma's statement is true.
\end{proof}
\subsection{End of proof of Theorem \ref{mainmal}}   

    Finally, recall that a set $X\subset\mathbb{Z}^d$ is \textit{supportive} if no cross $\{x\}\cup\{x\pm e_i\}$ intersects $X$ only by one of the points $x\pm e_i$. Define $$N_n(k) = \max_{X\text{ is }supportive,\,x,y\in X} |\{a\text{ of length n }|Q_n(x,a)=y, \lceil a_i/2\rceil\neq \lceil a_j/2\rceil\text{ for }|i-j|<k, \}|.$$ This is the biggest possible number of $k$-good sequences $a$ of length $n$ such that they all yield the same point $Q_n(x,a).$ Consider the set $X$ and the points $x,y$ such that the number $N_n(k)$ is reached and the set $S = \{a\text{ of length n }|Q_n(x,a)=y\}.$ Define the set $S_0 = \{a_0|a\in S\}$ and $l = |S_0|.$ Denote the $l$ numbers in $S_0$ by $s^1,\dots,s^l.$ For any number $s^i$ in $S_0$ there exists a sequence $a^i\in S$ starting with $s.$
    
    The sequences in the set $S$  start with $l$ different numbers $s^1,\dots,s^l\in S_0,$ and Lemma \ref{54} prevents any two values $\lceil s^i/2\rceil$ and $\lceil s^j/2\rceil$ from coinciding. For any number $s^i$ we have defined the value $\Psi(x,s^i):=(Q_n(x,a^i)-x,\epsilon_{s^i}).$  The inequality \eqref{prods} ensures that the sum $\sum_{i=1}^l \max(0,\Psi(x,s^i))\leq \sum_{i=1}^l|\Psi(x,s^i)|\leq 4r_n.$ Meanwhile, for any $s^i$ the Lemma \ref{4.5} ensures that the sequences $a$ starting with $s^i$ have the same members $a^i_m$ for any $$m<-\log_2\left(\frac{2^{-n}+2^{2-k}+2\max(0,\Psi(x,s^i))/r_n}{1-6\cdot 2^{-k}}\right)-1$$ Thus $$N_n(k)\leq \sum_{i=1}^l N_{n-m(i)}(k),$$ where $m(i)$ is defined as $$m(i):= \max\left(1,\left\lceil-\log_2\left(\frac{2^{-n}+2^{2-k}+2\max(0,\Psi(x,s^i))/r_n}{1-6\cdot 2^{-k}}\right)\right\rceil-1\right).$$ But $\sum |\Psi(x,s^i)|\leq 4r_n,$ and for $n,k>\log_2(d)+3$  $$\sum_{i=1}^l 2^{-m(i)} \leq 2\sum_{i=1}^l \frac{2^{-n}+2^{2-k}+2\max(0,\Psi(x,s^i))/r_n}{1-6\cdot 2^{-k}} \leq 2\cdot\frac{2^{1-n}d+2^{3-k}d+8}{1-6\cdot 2^{-k}} \leq 64.$$ Therefore, for $n,k>\log_2(d)+3$ we have $$2^{-7n}N_n(k)\leq 2^{-7n}\sum_{i=1}^l N_{n-m(i)}(k) = \sum_{i=1}^l 2^{-7m(i)}2^{-7(n-m(i))}N_{n-m(i)}(k).$$ But $\sum_{i=1}^l 2^{-m(i)}\leq64.$ Therefore, $\sum_{i=1}^l 2^{-7m(i)} \leq 1,$ and $$2^{-7n}N_n(k)\leq \sum_{i=1}^l 2^{-7m(i)}2^{-7(n-m(i))}N_{n-m(i)}(k) \leq$$$$\leq \left(\sum_{i=1}^l 2^{-7m(i)}\right)\max_{w<n}2^{-7(n-w)}N_{n-w}(k)\leq \max_{w<n}2^{-7(n-w)}N_{n-w}(k).$$ Therefore, for any fixed $k>\log_2(d)+3$ the sequence $2^{-7n}N_n(k)$ is bounded in $n.$

    Finally, we notice that there are at least $2^n(d-k)^n$ sequences $a$ of length $n$ such that $\lceil a_i/2\rceil\neq \lceil a_j/2\rceil\text{ for }|i-j|<k.$ In addition, no point $y$ can be obtained from more than $2^{7n}C(k)$ such sequences $a$. Then the total number of different points $Q_n(a)$ is at least $\left(\frac{2d-2k}{2^7}\right)^nC(k)^{-1}.$ Fixing $k = [\log_2(d)]+4,$ we find that for $d\geq 2^7$ the total number of different points $Q_n(a)$ is at least $\left(\frac{d}{2^7}\right)^n\wt{C}(d).$ 

    On the other hand, $r_n=2^{n+2}-2,$ and the points $Q_n(x,a)$ are located at a distance of $O(2^n)$ from the point $x.$ Therefore, the discrete dimension $\dim(X)$ satisfies the following inequality: $$\dim(X) = \limsup_{r\to\infty} \frac{\ln|X\cap B_r(x)|}{\ln r} \geq \limsup \frac{\ln|X\cap B_r(x)}{\ln r}\frac{\ln\left(\frac{d}{2^7}\right)}{\ln(2)} \geq  \frac{\ln(d)}{\ln(2)} - 7 = \log_2(d)-7,$$ as stated in the Theorem \ref{mainmal}.
\end{proof}

\end{document}